\documentclass[oneside]{amsart}

\usepackage{amsmath}
\usepackage{amssymb}
\usepackage[a4paper]{geometry}
\usepackage{amsthm}

\newtheorem{theorem}{Theorem}
\newtheorem{proposition}{Proposition}

\newtheorem{definition}{Definition}
\newtheorem{corollary}{Corollary}
\theoremstyle{remark}
\newtheorem{remark}{Remark}

\DeclareMathOperator{\ran}{ran}

\begin{document}
\title[Some remarks on Hilbert representations of posets]
    {Some remarks on Hilbert representations of posets}

\author{V. Ostrovskyi}
\address{Institute of Mathematics, Tereshchenkivska str. 3, Kiev, Ukraine}
\curraddr{} \email{vo@imath.kiev.ua}
\thanks{Supported by Theme no. 20 "Evolution and spectral
problems of modern mathematical physics" of the program
"Mathematical modeling of physical and mechanical processes in a
highly inhomogeneous environment"}

\author{S. Rabanovich}
\address{Institute of Mathematics, Tereshchenkivska str. 3, Kiev, Ukraine}
\email{slavik@imath.kiev.ua}
\thanks{}
\dedicatory{To 70-th birthday of our Teacher Yu.S.Samoilenko}

\date{DD/MM/2010}

\begin{abstract}
For a certain class of finite posets, we prove that all their irreducible orthoscalar representations are finite-dimensional and describe those, for which there exist essential (non-degenerate) irreducible orthoscalar representations.
\end{abstract}

\maketitle

\section{Introduction}
Let $S$ be a finite partially ordered set (poset). A
representation of $S$ in a linear space $V$ is a collection of
subspaces $V_g$, $g\in S$, for which $V_g\subset V_h$ as $g<h$,
and one consider the representations $V_g$ in $V$  and $V_g'$ in
$V'$ to be equivalent if there exists an invertible operator
$T\colon V\to V'$, such that $V_g=TV_g$, $g\in S$. Representations
of posets have been extensively studied by A.V.Roiter and his
collaborators (see \cite{naz-roi,kleiner} and others), in
particular, classes of posets of finite type, tame type and wild
type were described.

In the case of a Hilbert space $H$, a representation of $S$ is a
collection of closed subspaces $H_g$, $g\in S$, for which
$H_g\subset H_h$ as $g<h$, and they are studied up to a unitary
equivalence: representations $H_g$ in $H$  and $H_g'$ in $H'$ are
equivalent if there exists a \emph{unitary} operator $U\colon H\to
H'$, such that $H_g=UH_g$, $g\in S$. It appeared that posets of
(Hilbert) tame type ($*$-tame type) have very simple structure
\cite{kostya_brasil}: they are chains or semichains. In
Section~\ref{sec:tame} we provide a short description of them
(Section~\ref{sec:tame_desc}). Also, in Section~\ref{sec:onepar}
we introduce and describe the class of unitarily one-parameter
poset (for them, the continuous series of irreducible
representations naturally depend on a single parameter) and
calculate the spectrum of a linear combination of the
corresponding projections (Section \ref{sec:spectrum}).

On the other hand, it was discovered in several recent papers (\cite{krs_2002,kru_roi,alb_ost_sam} and others) that in the case of primitive posets, an additional condition of \emph{orthoscalarity} (see Section~\ref{sec:ortho} for the  definition of orthoscalar representations) leads to the results very similar to the ones in a linear representations theory. Moreover, it was shown in \cite{kostya_uni,kostya_uni2} that this similarity can be extended to some cases of non-primitive posets by using unitarization technique: there exists a correspondence between classes of linear representations in $V$ and orthoscalar representations in $H$.

Therefore, it is still a problem to develop results on orthocalar representations of posets in the non-primitive case. In this paper, we study orthoscalar representations of the class of posets which can be decomposed into a union of two unitarily one-parameter posets. We start with the simplest example of the primitive $(1,1,1,1)$ poset (Section~\ref{sec:four}), orthoscalar representations of which are four-tuples of projections in $H$ whose linear combination is a scalar operator. Here we summarize some results of \cite{ost-sam99, kyrych, yusenko-four}, and give explicit formulas for representations similar to the ones established in \cite{ost-sam06}. The main result is that any irreducible orthoscalar representation of any poset which is a union of two one-parameter posets, is finite-dimensional (Section~\ref{sec:main}).

In Section~\ref{sec:examples} we consider examples of such posets
and their orthoscalar representations.

\section{Posets of $*$-tame type}\label{sec:tame}
\subsection{Description of $*$-tame posets}\label{sec:tame_desc}
Recall that given a finite poset $S$, its (Hilbert) representation is a collection of subspaces $S\ni g\mapsto H_g$ of closed subspaces of some Hilbert space $H$, such that $H_g\subset H_h$ for $g<h$, $g,h\in S$. Obviously, each subspace $H_g$ is uniquely determined by an orthogonal projection $P_g$ onto $H_g$, $g\in S$, therefore,  representations of a poset $S$ are described by representations of a $*$-algebra generated by projections $P_g$, $g\in S$, such that $P_gP_h=P_g$, $g<h$, and vice versa. Notions of indecomposable, irreducible representations and unitary equivalence of representations are standard for representations of $*$-algebras and thus can be applied to Hilbert representations of posets as well.

For Hilbert representations of posets,  it is well-known that
$(1)$ and $(1,1)$ are posets of tame type, and the posets $(1,2)$
and $(1,1,1)$ are of $*$-wild type.

\begin{proposition}
 A poset $S$ is of tame type if and only if its width is 1 or 2 and $S$ does not contain the $(1,2)$ poset.
\end{proposition}

\begin{proof}
 If $S$ contains the $(1,2)$ poset, it is evidently of $*$-wild type . If the width of $S$ is 3 or more, it contains the $(1,1,1)$ poset and is again of $*$-wild type \cite{ost-sam99}.

It $S$ is of width $1$, the corresponding algebra is commutative, therefore $S$ is of tame type.

Let $S$ be of width $2$ and does not contain $(1,2)$. It is easy
to see that in this case $S = S_1\cup\dots\cup S_m$ such that each
$S_i$ is either $(1)$ or $(1,1)$, $i=1$, \dots, $m$, and
$S_j>S_k$, $j<k$. The latter means that and for any $f\in S_j$,
$g\in S_k$ we have $g<f$. But in this case any representation of
$S$ is a tensor product of representations of $S_i$, $i=1$, \dots,
$m$.
\end{proof}

It is easy to see that any irreducible representation of a poset
of $*$-tame type is one- or two-dimensional (see, e.g.,
\cite{kostya_brasil}), moreover, two-dimensional irreducible
representations exist if and only if the poset is of width 2. As
noticed above, any poset $S$ of tame type can be represented as
\begin{equation}\label{tame_decompos}
 S = S_1\cup\dots\cup S_m,
\end{equation}
such that each $S_i$ is either $(1)$ or $(1,1)$, $i=1$, \dots,
$m$, and $S_j>S_k$, $j<k$.

\subsection{One-parameter posets}\label{sec:onepar}
We introduce the class of one-parameter posets. For these posets, the set of irreducible representaions consists of a finite number of one-dimesional and a one-parameter continuous family of two-dimensional representations.

\begin{definition}
 We say that $S$ is a (unitarily) one-parameter poset if it is of tame type and in its decomposition \eqref{tame_decompos} exactly one set $S_i$ is of width 2.
\end{definition}

In other words, unitarily one-parameter posets are those having the following Hasse diagrams:
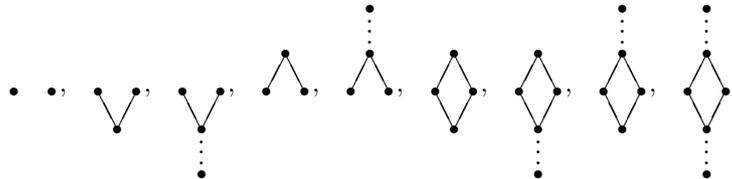
\[
\vrule height 12mm depth 12mm width 0mm
\setlength{\unitlength}{0.5mm}
\begin{picture}(12,0)(0,0)
  \put(0,0){\circle*{2}}
  \put(10,0){\circle*{2}}
\end{picture}
,\quad
\begin{picture}(12,10)(0,10)
  \put(5,0){\circle*{2}}
  \put(10,10){\circle*{2}}
  \put(0,10){\circle*{2}}
  \put(5,0){\line(-1,2){5}}
  \put(5,0){\line(1,2){5}}
\end{picture}
,\quad
\begin{picture}(12,10)(0,10)
  \put(5,-12){\circle*{2}}
  \put(5,-4){\makebox(0,0)[cc]{$\vdots$}}
  \put(5,0){\circle*{2}}
  \put(0,10){\circle*{2}}
  \put(10,10){\circle*{2}}
  \put(5,0){\line(1,2){5}}
  \put(5,0){\line(-1,2){5}}
\end{picture}
,\quad
\begin{picture}(12,0)(0,0)
  \put(5,10){\circle*{2}}
  \put(0,0){\circle*{2}}
  \put(10,0){\circle*{2}}
  \put(5,10){\line(1,-2){5}}
  \put(5,10){\line(-1,-2){5}}
\end{picture}
,\quad
\begin{picture}(12,0)(0,-10)
  \put(5,12){\circle*{2}}
  \put(5,8){\makebox(0,0)[cc]{$\vdots$}}
  \put(5,0){\circle*{2}}
  \put(0,-10){\circle*{2}}
  \put(10,-10){\circle*{2}}
  \put(5,0){\line(-1,-2){5}}
  \put(5,0){\line(1,-2){5}}
\end{picture}
,\quad
\begin{picture}(12,0)(0,0)
  \put(0,0){\circle*{2}}
  \put(10,0){\circle*{2}}
  \put(5,10){\circle*{2}}
  \put(5,-10){\circle*{2}}
  \put(5,10){\line(-1,-2){5}}
  \put(5,10){\line(1,-2){5}}
  \put(5,-10){\line(1,2){5}}
  \put(5,-10){\line(-1,2){5}}
\end{picture}
,\quad
\begin{picture}(12,0)(0,0)
  \put(5,-14){\makebox(0,0)[cc]{$\vdots$}}
  \put(0,0){\circle*{2}}
  \put(10,0){\circle*{2}}
  \put(5,10){\circle*{2}}
  \put(5,-10){\circle*{2}}
  \put(5,-22){\circle*{2}}
  \put(0,0){\line(1,2){5}}
  \put(0,0){\line(1,-2){5}}
  \put(10,0){\line(-1,2){5}}
  \put(10,0){\line(-1,-2){5}}
\end{picture}
,\quad
\begin{picture}(12,0)(0,0)
  \put(5,18){\makebox(0,0)[cc]{$\vdots$}}
  \put(0,0){\circle*{2}}
  \put(10,0){\circle*{2}}
  \put(5,10){\circle*{2}}
  \put(5,22){\circle*{2}}
  \put(5,-10){\circle*{2}}
  \put(0,0){\line(1,2){5}}
  \put(0,0){\line(1,-2){5}}
  \put(10,0){\line(-1,2){5}}
  \put(10,0){\line(-1,-2){5}}
\end{picture}
,\quad
\begin{picture}(12,0)(0,0)
  \put(5,18){\makebox(0,0)[cc]{$\vdots$}}
  \put(5,-14){\makebox(0,0)[cc]{$\vdots$}}
  \put(0,0){\circle*{2}}
  \put(10,0){\circle*{2}}
  \put(5,10){\circle*{2}}
  \put(5,-10){\circle*{2}}
  \put(5,22){\circle*{2}}
  \put(5,-22){\circle*{2}}
  \put(0,0){\line(1,2){5}}
  \put(0,0){\line(1,-2){5}}
  \put(10,0){\line(-1,2){5}}
  \put(10,0){\line(-1,-2){5}}
\end{picture}
.
\]

\begin{remark}
 In the theory of linear (non-unitary) representations of poset, the class of one-parameter posets is defined in other terms and differs from the one introduced above. Below, speaking of one-parameter posets we mean unitarily one-parameter posets unless specified explicitly.
\end{remark}

The description of irreducible representations given in \cite{kostya_brasil} in the case of one-parameter posets can be specified as follows.

Let $S$ be a one-parameter poset, and let $k$ be the unique index, for which $S_k$ in \eqref{tame_decompos} is of width 2. Then $S_j$, $j\ne k$ consists of a single element $g_j$, while $S_k$ consists of two elements $g_{k,1}$,  $g_{k,2}$.

\begin{proposition}\label{prop:reps-noos}
Any irreducible representation of $S \ni g\mapsto P_g$ has dimension one or two. There exists a finite number of one-dimensional irreducible representations, $P_g=p_g\in\{0,1\}$, where
\[
 p_1\le\dots\le p_{k-1}\le \frac{p_{k,1}+p_{k,2}}2\le p_{k+1}\le\dots\le p_m,
\]
and a one-parameter family of two-dimensional irreducible
non-equivalent representations,
\begin{gather*}
 P_1=\dots =P_{k-1}=0,
\\
 P_{k,1} = \begin{pmatrix}1&0\\0&0\end{pmatrix},\quad
 P_{k,2}= \begin{pmatrix}\tau&\sqrt{\tau(1-\tau)}\\ \sqrt{\tau(1-\tau)}&1-\tau\end{pmatrix},\qquad 0<\tau<1,
\\
 \quad P_{k+1}=\dots=P_m=I.
\end{gather*}
\end{proposition}

\subsection{Linear combinations of projections}\label{sec:spectrum}
The latter proposition enables one to obtain in a standard way  a
spectral decomposition of arbitrary (reducible) representation
into a direct sum or integral of irreducible ones. We use such a
decomposition to describe the spectrum of operator \(
 \sum_{g_s \in S}\alpha_s P_s,
\)
which will be used below.

\begin{proposition}\label{spectrum}
Let $S\ni g \mapsto P_g$ be a Hilbert representation of a one-parameter poset $S$. Denote
\[
 \sum_{g_s \in S}\alpha_s P_s.
\]
Then
\begin{align*}
 \sigma (A)&\subset \Delta = \Delta_d \cup \frac{\Sigma\pm (|\alpha_{k,1}-\alpha_{k,2}|,\alpha_{k,1}+\alpha_{k,2})}2 ,
\end{align*}
where
\begin{align}
\Delta_d& = \{0, \alpha_m, \alpha_{m-1}+\alpha_m,\dots, \alpha_{k+1}+\dots+\alpha_m,
\\
&\alpha_{k,1}+\alpha_{k+1}+\dots+\alpha_m, \alpha_{k,2}+\alpha_{k+1}+\dots+\alpha_m,\alpha_{k,1}+\alpha_{k,2}+\alpha_{k+1}+\dots+\alpha_m,
\notag
\\
& \alpha_{k-1}+\alpha_{k,1}+\alpha_{k,2}+\alpha_{k+1}+\dots+\alpha_m, \dots
\notag
\\
&\alpha_1+\dots +\alpha_{k-1}+\alpha_{k,1}+\alpha_{k,2}+\alpha_{k+1}+\dots+\alpha_m\}
\notag
\\
\Sigma&=\alpha_{k,1}+\alpha_{k,2} + 2 \sum_{j>k}\alpha_j.
\label{eq:sigma}
\end{align}
Moreover, the parts of the spectrum in the continuous area
corresponding to the plus and minus signs have the same type and
multiplicity. In particular, the number $(\Sigma + \lambda)/2$,
$|\alpha_{k,1}-\alpha_{k,2}|<\lambda <\alpha_{k,1}+\alpha_{k,2}$,
is an eigenvalue if and only if $(\Sigma - \lambda)/2$ is an
eigenvalue of the same multiplicity.
\end{proposition}

\begin{proof}
In the case of two-dimensional irreducible representation, by
routine calculations we obtain that the spectrum  $\sigma
(\sum_{g_s\in S}\alpha_s P_s)$ consists of two points,
\[
\lambda = \frac{(\alpha_{k,1} + \alpha_{k,2}) \pm \sqrt{\alpha_{k,1}^2 + \alpha_{k,2}^2 +2\alpha_{k,1}\alpha_{k,2}( 2\tau -1 )}}2 + \sum_{j>k} \alpha_j.
\]
Since the parameter $\tau$ can take arbitrary value in $(0,1)$, in the general case these representaions  give two symmetric with respect to the point $ \frac{(\alpha_{k,1} + \alpha_{k,2}) }2 + \sum_{j>k} \alpha_j$ segments,
\[
 \frac12 \Bigl((\alpha_{k,1} + \alpha_{k,2}) \pm (|\alpha_{k,1} - \alpha_{k,2}|,\alpha_{k,1} + \alpha_{k,2})\Bigr) + \sum_{j>k} \alpha_j.
\]
The rest of the possible points of $\sigma (\sum_{g_s \in
S}\alpha_s P_s)$ arise from one-dimensional representations.
\end{proof}


Given $\lambda\in \sigma(A)$, $\lambda\in(\Sigma\pm
(|\alpha_{k,1}-\alpha_{k,2}|,\alpha_{k,1}+\alpha_{k,2}))/2$, we
can restore the corresponding projections. Indeed, let
\[
\tau =
\frac{(\sum_{j>k} \alpha_j + \alpha_{k,1}-\lambda)(\sum_{j>k} \alpha_j + \alpha_{k,2}-\lambda) }{\alpha_{k,1}\alpha_{k,2}}.
\]
If $\lambda\in \sigma(A)$ is an eigenvalue corresponding to the
continuous part of $\Delta$, then in the corresponding eigenbasis
of $A$, we have
\[
 P_{k,1}=\frac12\begin{pmatrix}1+\epsilon_1&-\sqrt{1-\epsilon_1^2}\\-\sqrt{1-\epsilon_1^2}&1-\epsilon_1\end{pmatrix}, \quad
 P_{k,2}=\frac12\begin{pmatrix}1+\epsilon_2&\sqrt{1-\epsilon_2^2}\\ \sqrt{1-\epsilon_2^2}&1-\epsilon_2\end{pmatrix}, \quad
\]
where
after routine calculations,
\[
   \epsilon_1= \frac{ 2 \mu^2-(2\mu- \alpha_{k,1})(\alpha_{k,1}+\alpha_{k,2})   }
{ \alpha_{k,1}(2\mu -\alpha_{k,1}- \alpha_{k,2})  },\quad
   \epsilon_2= \frac{ 2 \mu^2-(2\mu- \alpha_{k,2})(\alpha_{k,1}+\alpha_{k,2})   }
{ \alpha_{k,2}(2\mu -\alpha_{k,1}- \alpha_{k,2})  }.
\]
Here $\mu =\lambda - \sum_{j>k} \alpha_j $.

\section{Orthoscalar representations of finite posets}

\subsection{Definition of orthoscalarity}\label{sec:ortho}
Let $S$ be a finite poset, $S\ni g \mapsto P_g$ be a collection of orthoprojections which form its representation, i.e. $P_gP_h=P_g$, $g<h$.

We use the term \emph{character} for a positive function on $S$, $S\ni g \mapsto \alpha_g >0$.
\begin{definition}
  We say that a representation $S\ni g \mapsto P_g$,  $P_gP_h=P_g$, $g<h$, is orthoscalar with a character $\alpha =(\alpha_g)_{g\in S}$, if
\[
 \sum_{g\in S}\alpha_g P_g =I.
\]
\end{definition}

Orthoscalar representations of primitive posets (in terms of orthoscalar representations of graphs or quivers) were studied in \cite{krs_2002,kru_roi,alb_ost_sam} and other papers, some results for the non-primitive case are obtained in \cite{kostya_uni,kostya_uni2} and others.

Notice the following simple properties of orthoscalar representations.

1. If $\sum_{g\in S} \alpha_g <1$, there are no representations.

2. If $\sum_{g\in S} \alpha_g =1$, then all $P_g=I$, $g\in S$.

3. If $\alpha_g>1$ for some $g\in S$, then $P_g=0$.

4. If $\alpha_g=1$ for some $g\in S$, then in any irreducible representation either $P_g=0$ or $P_g=I$.

To exclude these degenerated cases, in what follows we assume that $0<\alpha_g<1$, $g\in S$, and $\sum_{g\in S} \alpha_g >1$.

In this paper we study the class of finite posets $S$, such that $S=S_1\cup S_2$, where $S_1$ and $S_2$ are unitarily one-parameter posets of tame type. We admit that some elements of $S_1$ can be comparable with some elements of $S_2$, however, we do not use such relations, and they should be taken into account to narrow the result obtained without them.

\subsection{Orthoscalar four-tuples of projections}\label{sec:four}
The simplest case of poset of such kind is the $(1,1,1,1)$ poset --- primitive poset, consisting of four elements, any two of which are non-comparable. An orthoscalar representation of this poset with character $\alpha=(\alpha_1,\alpha_2,\alpha_3,\alpha_4)$ is a four-tuples of projections, $P_1$, \dots, $P_4$, in some Hilbert space $H$, for which
\begin{equation}\label{eq:4tuple}
 \alpha_1P_1+\alpha_2P_4 +\alpha_3P_3 +\alpha_4P_4=I.
\end{equation}
Such four-tuples have been studied in \cite{ost-sam99, ost-sam06,kyrych,yusenko-four} and others. In particular, the following theorem has been proved (see also \cite{ost05}).

\begin{theorem}
 Any orthoscalar irreducible four-tuple of projections is finite-dimensional.
\end{theorem}

Here we give an independent proof of this fact, which involves constructions which we will apply in a more general case.

\subsubsection{Case of $\alpha_1+\alpha_2+\alpha_3+\alpha_4=2$. Continuous series}
\begin{proposition}
 Let $P_1$, \dots, $P_4$ be an irreducible family of projections in $H$ satisfying \eqref{eq:4tuple}, for which $\ker P_j\cap \ker P_k=\{0\}$, $\ran P_j\cap \ran P_k=\{0\}$, $\ker P_j\cap  \ran P_k=\{0\}$,  $j\ne k$. Then $\alpha_1+\alpha_2+\alpha_3+\alpha_4=2$ and $\dim H=2$.
\end{proposition}

\begin{proof}
 Introduce operators $A_1=\alpha_1P_1+\alpha_2P_2$, $A_2=\alpha_3P_3+\alpha_4P_4$.
  The orthoscalarity condition means $A_1+A_2=I$.  The conditions that the kernels and ranges of the projections are zero imply, due to the structure theorem for a pair of projections, that the space $H$ can be decomposed as $H=\mathcal H\oplus \mathcal H =\mathbb C^2\otimes \mathcal H$ so that
\begin{equation}\label{eq:a1a2four}
A_1 = \frac{\alpha_1+\alpha_2}{2}I +\begin{pmatrix}1&0\\0&- 1\end{pmatrix}\otimes C_1, \quad A_2 = \frac{2-\alpha_1-\alpha_2}{2}I+\begin{pmatrix}-1&0\\0& 1\end{pmatrix} \otimes C_1,
\end{equation}
where $\frac{|\alpha_1-\alpha_2|}{2}I<C_1<\frac{\alpha_1+\alpha_2}{2}I$,
is a self-adjoint operator in $\mathcal H$. Applying the same structure theorem to $P_3$, $P_4$, we conclude that $A_2$ can be represented (probably for another decomposition $H =\mathbb C^2\otimes \mathcal H'$) as
\[
 A_2=\frac{\alpha_3+\alpha_4}{2}I +\begin{pmatrix}-1&0\\0& 1\end{pmatrix} \otimes C_2 , \quad \frac{|\alpha_3-\alpha_4|}{2}I<C_2<\frac{\alpha_3+\alpha_4}{2}I.
\]
Comparing this to \eqref{eq:a1a2four}, we have
$ \alpha_1+\alpha_2+\alpha_3+\alpha_4 =2$.

It is easy to see that the operator $
\Bigl(A_1-\frac{\alpha_1+\alpha_2}{2}I\Bigr)^2 = I\otimes C_1^2$
commutes with $P_1$ and $P_2$. The same way, the operator $(A_2 -
\frac{\alpha_3+\alpha_4}{2}I)^2= I\otimes C_1^2$ commutes with
$P_3$, $P_4$. Therefore, $I\otimes C_1^2$ commutes with $P_1$,
\dots, $P_4$, and therefore, is a scalar operator in an
irreducible representation. Thus, $C_1=cI$ for some
\[
c\in  \Bigl(\frac{|\alpha_1-\alpha_2|}{2},\frac{\alpha_1+\alpha_2}{2}\Bigr)\cap \Bigl(\frac{|\alpha_3-\alpha_4|}{2},\frac{\alpha_3+\alpha_4}{2}\Bigr),
\]
and the irreducibility implies $\mathcal H=\mathbb C$.
\end{proof}

\begin{remark}
 One can obtain explicit formulas for the corresponding two-dimensional representations. Write the projections $P_1$, $P_2$ in the form
\begin{align}\label{eq:4tuple-twodim}
 P_1&=\frac12 \begin{pmatrix}1+\lambda_1&\sqrt{1-\lambda_1^2}\\ \sqrt{1-\lambda_1^2}&1-\lambda_1\end{pmatrix},
&
 P_2&=\frac12 \begin{pmatrix}1+\lambda_2&-\sqrt{1-\lambda_2^2}\\-\sqrt{1-\lambda_2^2}&1-\lambda_2\end{pmatrix},
\end{align}
with
\[
\lambda_1=\frac{\alpha_1^2-\alpha_2^2 +4c^2} {4c\alpha_1}, \quad \lambda_2=\frac{\alpha_2^2-\alpha_1^2 +4c^2} {4c\alpha_2},
\]
so that $A_1=\alpha_1P_1+\alpha_2P_2 = c\bigl(\begin{smallmatrix}1&0\\0&-1\end{smallmatrix}\bigr)$. The projections $P_3$ and $P_4$ can be represented as
\begin{equation}\label{eq:4tuple-twodim2}
P_3=\frac12 \begin{pmatrix}1+\lambda_3&\gamma\sqrt{1-\lambda_3^2}\\ \bar\gamma \sqrt{1-\lambda_3^2}&1-\lambda_3\end{pmatrix},
\quad
 P_4=\frac12 \begin{pmatrix}1+\lambda_4&-\gamma\sqrt{1-\lambda_4^2}\\-\bar\gamma\sqrt{1-\lambda_4^2}&1-\lambda_4\end{pmatrix},
\end{equation}
with
\[
\lambda_3=\frac{\alpha_3^2-\alpha_4^2 +4c^2} {4c\alpha_3}, \quad \lambda_4=\frac{\alpha_4^2-\alpha_3^2 +4c^2} {4c\alpha_4},\quad |\gamma|=1.
\]
Therefore, the set of two-dimensional irreducible representations, for which $\ker P_j\cap \ker P_k=\{0\}$, $\ran P_j\cap \ran P_k=\{0\}$, $\ker P_j\cap  \ran P_k=\{0\}$,  $j\ne k$, is described by two continuous parameters, $c$ and $\gamma$.
\end{remark}

\subsubsection{Case $\alpha_1+\alpha_2+\alpha_3+\alpha_4=2$. Discrete series}
Now consider the case where at least one of the subspaces $\ker P_j\cap \ker P_k$, $\ran P_j\cap \ran P_k$, $\ker P_j\cap  \ran P_k$, $\ran P_j\cap  \ker P_k$, $j\ne k$, is nonzero. We obviously can assume $j=1$, $k=2$. Introduce sets
\begin{align*}
 \Delta_{1,d}&=\{0,\alpha_1,\alpha_2,\alpha_1+\alpha_2\},
&
 \Delta_{1,c}&=(0,\alpha_1)\cup (\alpha_2,\alpha_1+\alpha_2),
\\
 \Delta_{2,d}&=\{0,\alpha_3,\alpha_4,\alpha_3+\alpha_4\},
&
 \Delta_{2,c}&=(0,\alpha_3)\cup (\alpha_4,\alpha_3+\alpha_4),
\\
 \Delta_{1}&=[0,\alpha_1]\cup [\alpha_2,\alpha_1+\alpha_2],
&
 \Delta_{2}&=[0,\alpha_3]\cup [\alpha_4,\alpha_3+\alpha_4],
\end{align*}
so that $\Delta_{1}= \Delta_{1,d}\cup \Delta_{1,c}$, $\Delta_{2}= \Delta_{2,d}\cup \Delta_{2,c}$.

Then there exists a number $\lambda_0\in\Delta_{1,d}$ which is an eigenvalue of the operator $A_1$. Let $f_0$ be the corresponding unit eigenvector. Since $A_1+A_2=I$, $f_0$ is also an eigenvector of $A_2$, $A_2f_0 =\mu_0 f_0$, where $\mu_0=1-\lambda_0$. The following two cases can arise.

(i) $\mu_0\in \Delta_{2,d}$. Then the space spanned by $f_0$ is invariant w.r.t. $P_1$, \dots, $P_4$, and due to the irreducibility, is the whole $H$, $\dim H=1$, and
\begin{equation}\label{eq:4tuple-onedim}
P_1=\delta_1, \quad P_2=\delta_2, \quad P_3=\delta_3, \quad P_4 = \delta_4, \qquad \delta_1,\delta_2,\delta_3,\delta_4\in \{0,1\}.
\end{equation}
Notice that in this case there exists such permutation $\sigma$ of indexes, that
\[
\alpha_{\sigma(1)}+\alpha_{\sigma(2)} = \alpha_{\sigma(3)}+\alpha_{\sigma(4)}=1.
\]

(ii) $\mu_0\in \Delta_{2,c}$. Then $\mu_1=\alpha_3+\alpha_4 - \mu_0$ is also an eigenvalue of $A_2$ with some unit eigenvector $f_1$.

In the latter case, since $A_1+A_2=I$, the vector $f_1$ is an
eigenvector of $A_1$ as well, $A_1f = (I-A_2)f = \lambda_1 f_1$,
$\lambda_1 = 1-\mu_1$. Since
$\alpha_1+\alpha_2+\alpha_3+\alpha_4=2$, one can see that
$\lambda_1\in\Delta_{1,d}$. Indeed, otherwise
$\lambda_1\in\Delta_{1,c}$ and $\alpha_1+\alpha_2 - \lambda_1
=\lambda_0 \in \Delta_{1.c}$ which contradicts the initial setting
$\lambda_0\in\Delta_{1,d}$. Therefore, $H$ is spanned by
$(f_0,f_1)$. The projections are
\begin{align}\label{eq:4tuple-twodim3}
P_1&=
\begin{pmatrix}
\delta_1&0\\0&1-\delta_1
\end{pmatrix},
&
P_3 &= \frac12
\begin{pmatrix}
1+\tau_1&\sqrt{1-\tau_1^2}\\ \sqrt{1-\tau_1^2}& 1-\tau_1
\end{pmatrix},
\\
\notag
P_2&=
\begin{pmatrix}
\delta_2&0\\0&1-\delta_2
\end{pmatrix},
&
P_4 &= \frac12
\begin{pmatrix}
1+\tau_2&-\sqrt{1-\tau_2^2}\\ -\sqrt{1-\tau_2^2}& 1-\tau_2
\end{pmatrix},
\end{align}
where $\delta_1,\delta_2\in \{0,1\}$ are defined from $\alpha_1\delta_1+\alpha_2\delta_2=\lambda_0\in\Delta_{1,d}$, and
\[
\tau_1=\frac{2\mu_0^2 - (2\mu_0 -\alpha_3)(\alpha_3+\alpha_4)}{\alpha_3(2\mu_0 -\alpha_3-\alpha_4)}
, \quad
\tau_2= \frac{2\mu_0^2 - (2\mu_0 -\alpha_4)(\alpha_3+\alpha_4)}{\alpha_4(2\mu_0 -\alpha_3-\alpha_4)}
,
\]
$\mu_0=1-\lambda_0\in\Delta_{2,c}$.

\subsubsection{Case $\alpha_1+\alpha_2+\alpha_3+\alpha_4\ne2$}
In the case where $\alpha_1+\alpha_2+\alpha_3+\alpha_4\ne2$,  for $\lambda_1$ there can be two possibilities.

(i) $\lambda_1\in \Delta_{1,d}$. Then the space spanned by $(f_0,f_1)$ is invariant w.r.t. $P_1$, \dots, $P_4$, and due to the irreducibility, is the whole $H$, $\dim H=2$, the projections are given by \eqref{eq:4tuple-twodim3}.

(ii) $\lambda_1\in \Delta_{1,c}$. In this case, $\lambda_2=\alpha_1+\alpha_2 - \lambda_1$ is also an eigenvalue of $A_1$ with some unit eigenvector $f_2$.

In the latter case, since $A_1+A_2=I$, the vector $f_2$ is an eigenvector of $A_2$ and we proceed as above.

Consider two sequences, $\lambda_0, \lambda_1,\dots$, and $\mu_0,\mu_1,\dots$, constructed from $\lambda_0$ by the following rules. For $j=2k$, $k\ge0$ let $\mu_j=1-\lambda_j$, $\mu_{j+1} = \alpha_3+\alpha_4 - \mu_j$, $\lambda_{j+1} = 1-\mu_{j+1}$, $\lambda_{j+2} = \alpha_1+\alpha_2 -\lambda_{j+1}$. We have
\begin{align*}
 \lambda_{2k}&=\lambda_0+ k(\alpha_1+\alpha_2+\alpha_3+\alpha_4 -2),
\\
 \lambda_{2k+1}&= \alpha_1+\alpha_2 -\lambda_0 - (k+1)(\alpha_1+\alpha_2+\alpha_3+\alpha_4 -2),
\\
 \mu_{2k}&=1-\lambda_0 - k(\alpha_1+\alpha_2+\alpha_3+\alpha_4 -2),
\\
 \mu_{2k+1}&= \lambda_0+ \alpha_3+\alpha_4 - 1 + k(\alpha_1+\alpha_2+\alpha_3+\alpha_4 -2), \quad k\ge 0.
\end{align*}
Then the arguments above imply that $\sigma(A_1) \subset (\lambda_k)_{k=0}^\infty$, $\sigma(A_2) \subset (\mu_k)_{k=0}^\infty$. Since $\sigma(A_1)\subset \Delta_1$, $\sigma(A_2)\subset\Delta_2$, then for $\alpha_1+\alpha_2+\alpha_3+\alpha_4 -2 \ne0$ only  finite number of $\lambda_k$ may belong to $\sigma (A_1)$, and the same number of $\mu_k$ may belong to $\sigma(A_2)$.

Therefore for $\alpha_1+\alpha_2+\alpha_3+\alpha_4 -2 \ne0$, $\sigma(A_1) = (\lambda_k)_{k=0}^m$, $\sigma(A_2) = (\mu_k)_{k=0}^m$, where $m\ge0$ is determined by the folowing conditions:
\begin{align}
\label{eq:4tupe-2m}
m=2l{:}\quad \lambda_k &\in \Delta_{1,c},\quad 1\le k \le m,
\\
\mu_k&\in\Delta_{2,c},\quad 0\le k \le m-1, \quad \mu_m\in\Delta_{2,d} ;\notag
\\
\label{eq:4tupe-2m+1}
m=2l+1{:}\quad \lambda_k &\in \Delta_{1,c}, \quad 1\le k \le m-1, \quad \lambda_m\in\Delta_{1,d},
\\
\mu_k&\in\Delta_{2,c},\quad 0\le k \le m. \notag
\end{align}
The dimension of the space $H$ is equal to $m+1$.

\subsubsection{Description of representations}
As we already shown, in the case where
$\alpha_1+\alpha_2+\alpha_3+\alpha_4 -2 =0$, or
$\alpha_1+\alpha_2+\alpha_3+\alpha_4 -2 \ne0$, $\dim H\le 2$, the
projections are given by the formulas \eqref{eq:4tuple-twodim},
\eqref{eq:4tuple-twodim2}, \eqref{eq:4tuple-onedim} or
\eqref{eq:4tuple-twodim3}. Now assume
$\alpha_1+\alpha_2+\alpha_3+\alpha_4 -2 \ne0$, $\dim H> 2$. In
order to give explicit formulas for the projections, let us
introduce projections in $\mathbb C^2$
\begin{equation}\label{eq:basicpair}
P_\tau = \frac12\begin{pmatrix} 1+\tau&\sqrt{1-\tau^2}\\ \sqrt{1-\tau^2}& 1-\tau\end{pmatrix}, \quad Q_\tau = \frac12\begin{pmatrix} 1+\tau&-\sqrt{1-\tau^2}\\ -\sqrt{1-\tau^2}& 1-\tau\end{pmatrix}, \qquad \tau\in (0,1).
\end{equation}

In the case $m=2l$, $l\ge 1$, the space $H=\mathbb C^{m+1}$ spanned by the joint eigenvectors $f_0$, \dots, $f_m$, of $A_1$ and $A_2$,  can be written as
$\mathbb C\oplus \underbrace{\mathbb C^2\oplus\dots\oplus\mathbb C^2}_{\mbox{$l$ times}}$, or as ${\underbrace{\mathbb C^2\oplus\dots\oplus\mathbb C^2}_{\mbox{$l$ times}}} \oplus \mathbb C$, so that the projections take the form
\begin{align}\label{eq:4tuple-2lplus1dim}
 P_1 &= \delta_1 \oplus P_{p_1}\oplus\dots\oplus P_{p_l}, & P_2& = \delta_2 \oplus Q_{q_1}\oplus\dots\oplus Q_{q_l},
&H&= \mathbb C\oplus \underbrace{\mathbb C^2\oplus\dots\oplus\mathbb C^2}_{\mbox{$l$ times}},
\\
\notag
 P_3 &= P_{r_0}\oplus\dots\oplus P_{r_{l-1}}\oplus \delta_3  , & P_4 &= Q_{s_0}\oplus\dots\oplus Q_{s_{l-1}}\oplus\delta_4,
&H&={\underbrace{\mathbb C^2\oplus\dots\oplus\mathbb C^2}_{\mbox{$l$ times}}} \oplus \mathbb C,
\end{align}
where $\delta_1,\delta_2,\delta_3,\delta_4 \in \{0,1\}$ are defined from the conditions
\[
 \alpha_1\delta_1+\alpha_2\delta_2 =\lambda_0\in \Delta_{1,d}, \quad \alpha_3\delta_3+\alpha_4\delta_4 = \mu_m\in\Delta_{2,d},
\]
and
\begin{align*}
 p_j&=\frac{2\lambda_{2j-1}^2 - (2\lambda_{2j-1} -\alpha_1)(\alpha_1+\alpha_2)}{\alpha_1(2\lambda_{2j-1} -\alpha_1-\alpha_2)},
 \quad q_j=\frac{2\lambda_{2j-1}^2 - (2\lambda_{2j-1} -\alpha_2)(\alpha_1+\alpha_2)}{\alpha_2(2\lambda_{2j-1} -\alpha_1-\alpha_2)},
\\
&\qquad 1\le j \le l,
\\
 r_j&=\frac{2\mu_{2j}^2 - (2\mu_{2j} -\alpha_3)(\alpha_3+\alpha_4)}{\alpha_3(2\mu_{2j} -\alpha_3-\alpha_4)},
 \quad s_j=\frac{2\mu_{2j}^2 - (2\mu_{2j} -\alpha_4)(\alpha_3+\alpha_4)}{\alpha_4(2\mu_{2j} -\alpha_3-\alpha_4)},
\\
& \qquad 0\le j \le l-1.
\end{align*}
The conditions \eqref{eq:4tupe-2m} are equivalent to
\begin{align*}
\alpha_1+\alpha_2+\alpha_3+\alpha_4 -2 &= (1-\lambda_0-\mu_{2l})/l,
\\
  \lambda_0+ \frac kl (1-\lambda_0-\mu_m) &\in \Delta_{1,c},
\\
\mu_m +\frac kl (1-\lambda_0-\mu_m)& \in \Delta_{2.c}, \qquad 1\le k\le l,
\end{align*}
where $\lambda_0\in \Delta_{1,d}
$, $\mu_m \in
\Delta_{2,d}
$ (total of 16 possibilities).

In the case $m=2l-1$, $l\ge 1$, the space $H$ spanned by $f_0$, \dots, $f_m$ can be written as
$\mathbb C\oplus {\underbrace{\mathbb C^2\oplus\dots\oplus\mathbb C^2}_{\mbox{$l-1$ times}}}\oplus\mathbb C$, or as ${\underbrace{\mathbb C^2\oplus\dots\oplus\mathbb C^2}_{\mbox{$l$ times}}} $, so that the projections take the form
\begin{align}\label{eq:4tuple-2ldim}
 P_1 &= \delta_1 \oplus P_{p_1}\oplus\dots\oplus P_{p_{l-1}}\oplus\delta_2,
\\
\notag
 P_2& = \delta_3 \oplus Q_{q_1}\oplus\dots\oplus Q_{q_{l-1}}\oplus \delta_4,
&H&= \mathbb C\oplus {\underbrace{\mathbb C^2\oplus\dots\oplus\mathbb C^2}_{\mbox{$l-1$ times}}}\oplus\mathbb C,
\\
\notag
 P_3 &= P_{r_0}\oplus\dots\oplus P_{r_{l-1}}  ,
\\
\notag
 P_4 &= Q_{s_0}\oplus\dots\oplus Q_{s_{l-1}},
&H&={\underbrace{\mathbb C^2\oplus\dots\oplus\mathbb C^2}_{\mbox{$l$ times}}},
\end{align}
where $\delta_1,\delta_2,\delta_3,\delta_4 \in \{0,1\}$ are defined from the conditions
\[
 \alpha_1\delta_1+\alpha_2\delta_2 =\lambda_0\in\Delta_{1,d}, \quad \alpha_1\delta_3+\alpha_2\delta_4 = \lambda_m\in\Delta_{1,d},
\]
and
\begin{align*}
 p_j&=\frac{2\lambda_{2j-1}^2 - (2\lambda_{2j-1} -\alpha_1)(\alpha_1+\alpha_2)}{\alpha_1(2\lambda_{2j-1} -\alpha_1-\alpha_2)},
 \\
  q_j&=\frac{2\lambda_{2j-1}^2 - (2\lambda_{2j-1} -\alpha_2)(\alpha_1+\alpha_2)}{\alpha_2(2\lambda_{2j-1} -\alpha_1-\alpha_2)},
&& 1\le j \le l-1,
\\
 r_j&=\frac{2\mu_{2j}^2 - (2\mu_{2j} -\alpha_3)(\alpha_3+\alpha_4)}{\alpha_3(2\mu_{2j} -\alpha_3-\alpha_4)},
 \\
  s_j&=\frac{2\mu_{2j}^2 - (2\mu_{2j} -\alpha_4)(\alpha_3+\alpha_4)}{\alpha_4(2\mu_{2j} -\alpha_3-\alpha_4)},
&& 0\le j \le l-1.
\end{align*}
The conditions \eqref{eq:4tupe-2m+1} are equivalent to
\begin{align*}
\alpha_1+\alpha_2+\alpha_3+\alpha_4 -2 &= \frac1l(\alpha_1+\alpha_2 - \lambda_1-\lambda_m),
\\
  \lambda_0+ \frac kl (\alpha_1+\alpha_2-\lambda_0-\lambda_m)& \in \Delta_{1,c}, \qquad 1\le k \le l-1,
\\
1-\lambda_0 -\frac kl (1-\lambda_0-\mu_{2l}) &\in \Delta_{2.c}, \qquad 0\le k\le l-1,
\end{align*}
where $\lambda_0,\lambda_m \in
\Delta_{1,d}
$ (total of 16 possibilities).

\subsection{Main theorem}\label{sec:main}
The main result of this paper is the following.

\begin{theorem}\label{th:main}
 Any irreducible orthoscalar representation of a finite poset $S$ such that $S$ can be decomposed into a union of two unitarily one-parameter sets as described above, is finite-dimensional.
\end{theorem}

\begin{proof}
Let $S\ni g \mapsto P_g$ be an irreducible orthoscalar representation of $S$ with character $\alpha =(\alpha_g)_{g\in S}$.

Introduce operators
\begin{equation}\label{eq:a1a2}
 A_1=\sum_{g\in S_1}\alpha_g P_g, \quad A_2=\sum_{g\in S_2}\alpha_g P_g.
\end{equation}
and let $\Delta_1$, $\Delta_2$ be the corresponding sets described by Proposition~\ref{spectrum}, so that $\sigma(A_1)\subset \Delta_1$, $\sigma(A_2)\subset \Delta_2$.
Then the orthoscalarity condition is equivalent to $A_1+A_2=I$.
We also write $\Sigma_1$ and $\Sigma_2$ for numbers \eqref{eq:sigma} corresponding to $S_1$ and $S_2$ respectively.

First, we show that $A_1$ has an eigenvalue. Let $\lambda_0\in \sigma(A_1)$. Then, since $A_1+A_2=I$, $\mu_0=1-\lambda_0 \in \sigma(A_2)\subset \Delta_2$.

For $\mu_0$ there can be two possibilities. If $\mu_0$ lies in the discrete part of $\Delta_2$,
then $\mu_0$ is an eigenvalue of $A_2$ and therefore, $\lambda_0$ is an eigenvalue of $A_1$.

If $\mu_0$ lies in the continuous part of $\Delta_2$, then $\mu_1=\Sigma_2-1 \in \sigma(A_2)$, and since $A_1+A_2=I$, $\lambda_1=1-\mu_1 \in \sigma(A_1)\subset \Delta_1$. If $\lambda_1$ belongs to the discrete part of $\Delta_1$, then it is an eigenvalue of $A_1$, otherwise $\lambda_2 = \Sigma_1 - \lambda_1 \in \sigma(A_1)$ etc.


Thus, we have the following sequence of numbers:
\begin{align}\label{eq:ds}
\lambda_0\to \mu_0& = 1-\lambda_0 \to \mu_1 = \Sigma_2 - \mu_0
\\
\notag
&\to \lambda_1=1-\mu_1 \to \lambda_2 = \Sigma_1 -\lambda_1 \to \mu_2 = 1-\lambda_2 \to \dots,
\end{align}
and we terminate this sequence as soon as $\lambda_k$ hits into the discrete part of $\Delta_1$ or $\mu_k$ hits into the discrete part of $\Delta_2$ which would mean that all the numbers $\lambda_k$ are eigenvalues of $A_1$, and $\mu_k$ are eigenvalues of $A_2$.
Introduce $\Lambda = \Sigma_1+\Sigma_2 -2$, then
simple calculations yield
\begin{align}\label{eq15}
 \lambda_{2k}&= \lambda_0+ k\Lambda,
\\
\notag \lambda_{2k+1} & = \Sigma_1 -\lambda_0- (k+1)\Lambda,
\\
\notag
 \mu_{2k}& = 1-\lambda_0-k\Lambda,
\\
\notag
 \mu_{2k+1}& =\Sigma_2 - 1 +\lambda_0 + k\Lambda, \quad k=0,1,\dots.
\end{align}
If $\Lambda\ne0$, these sequences are unbounded, therefore,
assuming $\lambda_k\in\sigma(A_1)$, $\mu_k \in \sigma(A_2)$ we
conclude that the sequence \eqref{eq:ds} terminates, therefore, it
consists of eigenvalues of $A_1$ and $A_2$. If $\Lambda=0$, then
$(2A_1-\Sigma_1)^2 =(2A_2 -\Sigma_2)^2$ commutes with all $P_g$,
$g\in S$, and due to the irreducibility is a scalar operator. Then
$\sigma(A_1) = 1-\sigma(A_2)$ consists of two points, which are
eigenvalues.

This way, we have shown that in the case where $\Lambda\ne0$ one can assume that $A_1$ has at least one eigenvalue in the discrete part of $\Delta_1$. Taking this eiginvalue as $\lambda_0$ in \eqref{eq:ds} and repeating the argument above, we conclue that the spectrum of $A_1$ consists of a finite number of eigenvalues, $\lambda_0$, \dots, $\lambda_n$. Moreover, similarly to the case of quadruples of projections considered in Section~\ref{sec:four} one can construct a series of corresponding eigenvectors $f_0$, \dots, $f_n$, span of which is an invariant subspace for all $P_g$, $g\in S$ and thus is the whole space $H$.

If $\Sigma_1+\Sigma_2 -2=0$, we have that either $\mu_0=1$ belongs to the discrete part of $\Delta_2$ and irreducible representation is one-dimensional, or $\mu_0=1$ belongs to the continuous part of $\Delta_2$, then $\lambda_1 = 2-\Sigma_2 = 2-\Sigma_1-\Sigma_2 +\Sigma_1 = \Sigma_1$ belongs to the discrete part of $\Delta_1$ and irreducible representation is two-dimensional.
\end{proof}

\begin{remark}
 The proof in fact establishes a method to describe all irreducible representations of $S$, their dimensions and explicit formulas for the projections.
\end{remark}

\begin{remark}
For the case where $\cap_{g\in S_1}\ker P_g \ne \{0\}$, the value of $\Lambda$ enables one to obtain a rough estimate for the dimension of irreducible representations: for dimension $k\ge 2$, one can see that $\Lambda>0$, $1- (k-1)\Lambda \ge0$, which implies $k\le \Lambda^{-1}+1$.
\end{remark}

\section{Essential posets}\label{sec:exact}
Let $S$ be a poset, and let $S\ni g\mapsto P_g$ be its orthoscalar representation,
\[
 \sum_{g\in S}\alpha_gP_g = I.
\]
If $P_h=0$ for some $h\in S$, then the corresponding term can be excluded from the sum above, and the family $P_g$, $g\ne h$ forms an orthoscalar representation of the $S\setminus h$ poset. The same way, if $P_h=I$ for some $h \in S$, then
\[
 \sum_{g\in S, g\ne h}\frac{\alpha_g}{1-\alpha_h}P_g =I,
\]
and the family $P_g$, $g\ne h$ forms an orthoscalar representation of the $S\setminus h$ poset.

Also, if $h<k$ and $P_h=P_k$, then the family $P_g$, $g\ne h$,
forms an orthoscalar repesentation of the $S\setminus h$ poset
with $\alpha_k$ replaced by $\alpha_k'=\alpha_h+\alpha_k$.

In all these  cases, the representation of $S$ is essentially determined by a representation of a smaller poset $S\setminus h$.

\begin{definition}
 We say that an orthoscalar representation $S\ni g \mapsto P_g$ of $S$ is essential, if $P_g\ne 0$,  $P_g\ne I$ for all $g\in S$, and $P_g\ne P_h$ for all $g, h\in S$, $g<h$. We say that a poset $S$ is essential if it possesses an irreducible essential  orthoscalar representation.
\end{definition}

\begin{theorem}\label{th:ess}
 Let $S$ be an essential poset which is a union of two unitarily one-parameter posets. Then $S$ is one of the following posets
\begin{gather*}
\vrule height0pt width0pt depth 8mm
 \setlength{\unitlength}{0.2mm}
 \begin{picture}(60,25)(0,0)
  \put(0,0){\circle*{2}}
  \put(20,0){\circle*{2}}
  \put(40,0){\circle*{2}}
  \put(60,0){\circle*{2}}
 \end{picture}
\ , \quad
 \begin{picture}(60,25)(0,0)
  \put(0,0){\circle*{2}}
  \put(20,0){\circle*{2}}
  \put(40,0){\circle*{2}}
  \put(60,0){\circle*{2}}
  \put(10,20){\circle*{2}}
  \put(0,0){\line(1,2){10}}
  \put(20,0){\line(-1,2){10}}
 \end{picture}
\ , \quad
\begin{picture}(60,0)(0,0)
  \put(0,0){\circle*{2}}
  \put(20,0){\circle*{2}}
  \put(40,0){\circle*{2}}
  \put(60,0){\circle*{2}}
  \put(10,-20){\circle*{2}}
  \put(0,0){\line(1,-2){10}}
  \put(20,0){\line(-1,-2){10}}
 \end{picture}
\ , \quad
  \begin{picture}(60,40)(0,0)
  \put(0,0){\circle*{2}}
  \put(20,0){\circle*{2}}
  \put(40,0){\circle*{2}}
  \put(60,0){\circle*{2}}
  \put(10,20){\circle*{2}}
  \put(10,40){\circle*{2}}
  \put(00,0){\line(1,2){10}}
  \put(20,0){\line(-1,2){10}}
  \put(10,20){\line(0,1){20}}
 \end{picture}
\ , \quad
 \begin{picture}(60,0)(0,0)
  \put(0,0){\circle*{2}}
  \put(20,0){\circle*{2}}
  \put(40,0){\circle*{2}}
  \put(60,0){\circle*{2}}
  \put(10,-20){\circle*{2}}
  \put(10,-40){\circle*{2}}
  \put(0,0){\line(1,-2){10}}
  \put(20,0){\line(-1,-2){10}}
  \put(10,-40){\line(0,1){20}}
 \end{picture} \ ,
 \quad
 \begin{picture}(60,25)(0,0)
  \put(0,0){\circle*{2}}
  \put(20,0){\circle*{2}}
  \put(40,0){\circle*{2}}
  \put(60,0){\circle*{2}}
  \put(10,20){\circle*{2}}
  \put(50,20){\circle*{2}}
  \put(0,0){\line(1,2){10}}
  \put(20,0){\line(-1,2){10}}
  \put(40,0){\line(1,2){10}}
  \put(60,0){\line(-1,2){10}}
 \end{picture} \ ,
\quad
\begin{picture}(60,0)(0,0)
  \put(0,0){\circle*{2}}
  \put(20,0){\circle*{2}}
  \put(40,0){\circle*{2}}
  \put(60,0){\circle*{2}}
  \put(10,-20){\circle*{2}}
  \put(50,-20){\circle*{2}}
  \put(0,0){\line(1,-2){10}}
  \put(20,0){\line(-1,-2){10}}
  \put(40,0){\line(1,-2){10}}
  \put(60,0){\line(-1,-2){10}}
 \end{picture}
 \ .
\end{gather*}
\end{theorem}

\begin{proof}
 We keep the notations used in the previous Section. First consider the case where $\Lambda=0$.
  Then any irreducible representation is one or two-dimensional. If $S$ possesses two elements $g<h$,
   then either $P_h=I$, or $P_g=0$, or $P_g=P_h$, i.e. irreducible orthoscalar representation is not essential.
    Therefore, for $\Lambda=0$, the only poset with essential irreducible representations
     is  $S =(1,1,1,1)$, the poset considered in Section~\ref{sec:four}.

>From now on, we assume $\Lambda\ne 0$. In this case we can assume that $\sigma(A_1)$ contains an eigenvalue $\lambda_0$ in the discrete part of $\Delta_1$. Then the argument used in the proof of Theorem~\ref{th:main} implies that there can be the following two possibilities.

(i). Dimension  $\dim H = n+1$ is even, $\lambda_n$ lies in the discrete part of $\Delta_1$, other eigenvalues $\lambda_1$, \dots, $\lambda_{n-1}$ lie in the continuous part of $\Delta_1$, all $\mu_k$, $k=0$, \dots, $n$, lie in the continuous part of $\Delta_2$.

(ii). Dimension  $\dim H = n+1$ is odd, eigenvalues $\lambda_1$, \dots, $\lambda_{n}$ lie in the continuous part of $\Delta_1$, eigenvalue $\mu_n$ lies in the discrete part of $\Delta_2$, all other $\mu_k$, $k=0$, \dots, $n-1$ lie in the continuous part of $\Delta_2$.

Consider the case (i). Let $h_1$, $h_2$ be a (unique) pair of
incomparable elements of $S_2$. Since all $\mu_k$, $k=0$, \dots,
$n$, lie in the continuous part of $\Delta_2$, we see from the
structure theorem for a pair of projections $P_{h_1}$, $P_{h_2}$,
that $P_h=0$ for any $h<h_1$, $h<h_2$, and $P_h=I$ for any
$h>h_1$, $h>h_2$. Therefore, an essential irreducible orthoscalar
representation of $S$ exists in even dimension only if $S_2$
consists of two incomparable points, $S_2=(1,1)$.

For the set $S_1$, we have the following. Let $g_1$, $g_2$ be a (unique) pair of incomparable elements in $S_1$. By the structure theorem for a pair of projections $P_{g_1}$, $P_{g_2}$ we decompose
\[
H = \mathbb C^1\oplus \mathbb C^2\oplus\dots\oplus \mathbb C^2 \oplus \mathbb C^1
\]
into invariant with respect to $P_g$, $g\in S_1$, irreducible subspaces. Then
\begin{align}\label{eq:pgh}
P_g& = \delta_1 \oplus I_2 \oplus\dots\oplus I_2 \oplus \delta_2, \quad g>g_1, g>g_2,
\\
\notag P_h& = \delta_3 \oplus 0_2 \oplus\dots\oplus 0_2 \oplus
\delta_4, \quad h<g_1, h<g_2,
\end{align}
 where $\delta_1,\delta_2,\delta_3,\delta_4\in \{0,1\}$. Moreover, to exclude the cases $P_g=I$ and $P_h=0$
 we assume that $\delta_1+\delta_2 < 2$, $\delta_3+\delta_4 > 0$. In each of the two invariant  one-dimensional blocks, the sum $P_{g_1}+P_{g_2}$ can take values 0, 1, or 2, and the following cases can arise.

1). In the both blocks the sum is 0. Then for $P_g$ in \eqref{eq:pgh} we can have  $\delta_1+\delta_2 =0$, or $\delta_1+\delta_2 =1$,
thus there can be at most two different nontrivial projections $P_{g_3}$, $P_{g_4}$, $g_3>g_1$, $g_3>g_2$, $g_4>g_3$. For $P_h$ we  have $\delta_3=\delta_4 =0$, thus there are no nonzero $P_h$.

2). In one block the sum is 0, and in the other one it is 1.  Then for $P_g$ in \eqref{eq:pgh} we have  $\delta_1+\delta_2 =1$, and there can be at most one nontrivial projection $P_{g_3}$,  $g_3>g_1$, $g_3>g_2$. For $P_h$ we again have $\delta_3=\delta_4 =0$, thus there are no nonzero $P_h$.

3). In the both blocks the sum is 1. Then for $P_g$ in \eqref{eq:pgh} we have  $\delta_1+\delta_2 =2$, and for $P_h$ we have $\delta_3=\delta_4 =0$, thus there are no $P_g\ne I$, $P_h\ne0$.

4). In the first block the sum is 0, in the second one the sum is 2, or in the first block the sum is 2,  in the second one the sum is 0. For $P_g$ in \eqref{eq:pgh} we have  $\delta_1+\delta_2 =1$, and for $P_h$ we have $\delta_3+\delta_4 =1$.   There can be at most one nontrivial projection $P_{h}$, $h<g_1$, $h<g_2$, and at most one nontrivial projection $P_{g}$, $g>g_1$, $g>g_2$.

5) In one block the sum is 2, in the other one the sum is 1. Then similarly to the case 2 there can be at most one nontrivial projection  $P_{h}$, $h<g_1$, $h<g_2$. For any $g>g_1$, $g>g_2$ we have $P_g=0$.

6) In the both blocks the sum is 2. Then similarly to the case 1 there can be at most two different nontrivial projections $P_{h_1}$, $P_{h_2}$, $h_1<g_1$, $h_1<g_2$, $h_2< h_1$. For any $g>g_1$, $g>g_2$ we have $P_g=0$.

Therefore, an essential irreducible orthoscalar representation of even dimension can exist only for the following posets (we use the notation from \cite{otr})
\begin{gather*}
a_1\ \setlength{\unitlength}{0.2mm}
 \begin{picture}(60,25)(0,0)
  \put(0,0){\circle*{2}}
  \put(20,0){\circle*{2}}
  \put(40,0){\circle*{2}}
  \put(60,0){\circle*{2}}
 \end{picture}
\ , \quad
a_2\
 \begin{picture}(60,25)(0,0)
  \put(0,0){\circle*{2}}
  \put(20,0){\circle*{2}}
  \put(40,0){\circle*{2}}
  \put(60,0){\circle*{2}}
  \put(10,20){\circle*{2}}
  \put(0,0){\line(1,2){10}}
  \put(20,0){\line(-1,2){10}}
 \end{picture}
\ , \quad
a_6 \  \begin{picture}(60,45)(0,0)
  \put(0,0){\circle*{2}}
  \put(20,0){\circle*{2}}
  \put(40,0){\circle*{2}}
  \put(60,0){\circle*{2}}
  \put(10,20){\circle*{2}}
  \put(10,40){\circle*{2}}
  \put(00,0){\line(1,2){10}}
  \put(20,0){\line(-1,2){10}}
  \put(10,20){\line(0,1){20}}
 \end{picture}
\ , \quad
a_8\ \begin{picture}(60,25)(0,0)
  \put(0,0){\circle*{2}}
  \put(20,0){\circle*{2}}
  \put(40,0){\circle*{2}}
  \put(60,0){\circle*{2}}
  \put(10,20){\circle*{2}}
  \put(10,-20){\circle*{2}}
  \put(0,0){\line(1,2){10}}
  \put(20,0){\line(-1,2){10}}
  \put(0,0){\line(1,-2){10}}
  \put(20,0){\line(-1,-2){10}}
 \end{picture} \ ,
\end{gather*}
and the posets dual to $a_2$ and $a_6$
\[
\vrule width0pt height0pt depth 8mm
\setlength{\unitlength}{0.2mm}
\begin{picture}(60,0)(0,0)
  \put(0,0){\circle*{2}}
  \put(20,0){\circle*{2}}
  \put(40,0){\circle*{2}}
  \put(60,0){\circle*{2}}
  \put(10,-20){\circle*{2}}
  \put(0,0){\line(1,-2){10}}
  \put(20,0){\line(-1,-2){10}}
 \end{picture}
\ , \quad
 \begin{picture}(60,0)(0,0)
  \put(0,0){\circle*{2}}
  \put(20,0){\circle*{2}}
  \put(40,0){\circle*{2}}
  \put(60,0){\circle*{2}}
  \put(10,-20){\circle*{2}}
  \put(10,-40){\circle*{2}}
  \put(0,0){\line(1,-2){10}}
  \put(20,0){\line(-1,-2){10}}
  \put(10,-40){\line(0,1){20}}
 \end{picture} \ .
\]
We show that the $a_8$ poset arising in the case 4  above is not in fact essential, i.e., any its irreducible orthoscalar representation is not essential.
Indeed, in the case 4 above assume that in the first one-dimensional block $P_{g_1}=P_{g_2}=0$, and in the second one $P_{g_1}=P_{g_2}=1$, then $\delta_1 =\delta_3=0$, $\delta_2=\delta_4=1$, and in essential representation  $\sigma(A_1)$ contains $0$ and $\alpha_g +\alpha_{g_1}+\alpha_{g_2} +\alpha_h$. Then the sequence \eqref{eq:ds} is
\begin{align*}
 \lambda_0&=0 \to \mu_0 =1 \to \mu_1 = \Sigma_2 -1 \to \lambda_1 =2-\Sigma_2\to \lambda_2 =\Lambda \to \dots
\\
&\to \mu_{2n+1} =\Sigma_2 -1 +n \Lambda \to \lambda_{2n+1} =
\alpha_g +\alpha_{g_1}+\alpha_{g_2} +\alpha_h.
\end{align*}
Since we have already shown that an essential  orthoscalar
representation of the $a_8$ poset must  have dimension more than
2, we conclude that $\lambda_2\in \sigma(A_1)$ and therefore
$\Lambda>0$. On the other side, since $\lambda_{2k+1} = \Sigma_1
-(k+1)\Lambda$, we have $\lambda_1>\lambda_3>\dots
>\lambda_{2n+1}$, therefore $\lambda_{2n+1}\le \lambda_1=\Sigma_1
-\Lambda<\alpha_g +\alpha_{g_1}+\alpha_{g_2}$.

In the case (ii) of odd dimension, similar arguments lead to the following posets:

\begin{gather*}
a_1\ \setlength{\unitlength}{0.2mm}
 \begin{picture}(60,25)(0,0)
  \put(0,0){\circle*{2}}
  \put(20,0){\circle*{2}}
  \put(40,0){\circle*{2}}
  \put(60,0){\circle*{2}}
 \end{picture}
\ , \quad
a_4\
 \begin{picture}(60,25)(0,0)
  \put(0,0){\circle*{2}}
  \put(20,0){\circle*{2}}
  \put(40,0){\circle*{2}}
  \put(60,0){\circle*{2}}
  \put(10,20){\circle*{2}}
  \put(50,20){\circle*{2}}
  \put(0,0){\line(1,2){10}}
  \put(20,0){\line(-1,2){10}}
  \put(40,0){\line(1,2){10}}
  \put(60,0){\line(-1,2){10}}
 \end{picture} \ ,
\end{gather*}
and the poset dual to $a_4$
\[
\vrule width0pt height0pt depth 8mm
\setlength{\unitlength}{0.2mm}
\begin{picture}(60,0)(0,0)
  \put(0,0){\circle*{2}}
  \put(20,0){\circle*{2}}
  \put(40,0){\circle*{2}}
  \put(60,0){\circle*{2}}
  \put(10,-20){\circle*{2}}
  \put(50,-20){\circle*{2}}
  \put(0,0){\line(1,-2){10}}
  \put(20,0){\line(-1,-2){10}}
  \put(40,0){\line(1,-2){10}}
  \put(60,0){\line(-1,-2){10}}
 \end{picture}
 \ .
\qedhere
\]
\end{proof}
In the following section we will show that all posets listed in
Theorem~\ref{th:ess} admit essential irreducible orthoscalar
representations, and therefore they are essential ones.

\section{Examples}\label{sec:examples}
As was shown above, the essential posets that are unions of two
unitarily one-parameter posets may have $4$, $5$ or $6$ elements.
Let $S=S_1\cup S_2$ be one of them with $S_1\cap S_2=\emptyset$.
Obviously $S$ and $S_2\cup S_1$ are isomorphic. Therefore we
consider only the posets in which $S_1$ has two comparable
elements. The operators $A_1$ and $A_2$ will be defined by the
formula (\ref{eq:a1a2}), where $P_g$ is the orthoprojection
corresponding to the element $g\in S$ in an essential orthscalar
irreducible representation of $S$.  Note that in all posets $a_2$,
$a_4$ and $a_6$ below the element $g_5$ (or $g_6$) is the maximal
element of $S_1$. Whence the ranges of $A_1$ and $P_5$ (or $P_6$)
coincide. Thus if $A_1$ is invertible, then $P_5=I$ (or $P_6=I$)
and hence the representation is not essential. Therefore $A_1$ is
singular for every essential representation of $a_2$, $a_4$,
$a_6$. Let $\Lambda = \Sigma_1+\Sigma_2 -2$.

1) \emph{Representations of $a_2$.} Let $a_2=S_1\cup S_2$,
$S_1=\{g_1,g_2,g_5\mid g_5>g_1,g_5>g_2\}$, $S_2=\{g_3,g_4\}$ and
orthoprojections $P_1,\dots , P_5$ form an essential orthoscalar
representation of $a_2$  with character $
\alpha=(\alpha_1,\alpha_2,\alpha_3,\alpha_4,\alpha_5).$ The
operator $A_1$ is singular so there exists $f_0$ such that $A_1
f_0=\lambda_0 f_0=0$. According to Theorem \ref{th:main},  the
sequence $\lambda_i$ obtained by the formulas (\ref{eq15})
consists of eigenvalues of $A_1$ and the sequence $\mu_j$ consists
of eigenvalues of $A_2$ for $i=0,1,\dots , m'$ and $j=0,1,\dots ,
m'$ with some positive $m'$. Let $f_0$, $f_1,\dots$, $f_{m'}$ be
the corresponding eigenvectors. This consequence can be obtain
from $f_0$ using special linear combinations of $P_i$. Let
$D_1(x)=P_1 + \phi(x) P_2$, $D_2(x)=P_3 + \psi(x) P_4$, where
\[
\phi(x)=\frac{\alpha_2(\lambda-\alpha_2-\alpha_5)}
{\alpha_1(\alpha_1+\alpha_5-\lambda)},\quad
\psi(x)=\frac{\alpha_4(\lambda-\alpha_4)}
{\alpha_3(\alpha_3-\lambda)}.
\]
Then $f_1=D_2(\mu_0)$, $f_{2j}=D_1(\lambda_{2j-1}),$
$f_{2j+1}=D_2(\mu_{2j})$. Assume that
\begin{equation}\label{eq30}\lambda_{2i} \in
(\alpha_5,\alpha_1+\alpha_5) \cup
(\alpha_2+\alpha_5,\alpha_1+\alpha_2+\alpha_5)
 \end{equation}
and
\begin{equation}\label{eq31}\mu_{2j-1} \in
(0,\alpha_3) \cup (\alpha_4,\alpha_3+\alpha_4),
\end{equation}
where $i,j=1,\dots, m$. There exist only two cases in which the
representation can be reconstructed: $\lambda_{2m+1} \in
\{\alpha_5,\alpha_1+\alpha_5,\alpha_2+\alpha_5\}$ with $m'=2m+1$
or $\mu_{2m} \in \{0,\alpha_3,\alpha_4\}$ with $m'=2m$. We
consider both cases in details.

({\textbf{i}}) $\lambda_{2m+1} \in
\{\alpha_5,\alpha_1+\alpha_5,\alpha_2+\alpha_5\}$. The subspace
\begin{equation}\label{eq32}\mbox{span}(f_0, f_1,f_2,
\dots, f_{2m+1})
\end{equation}
 is invariant
under the act of $P_1,\dots,P_5$ for $\alpha_1\ne \alpha_2$.
Operators $P_1$, $P_2,\dots,$ $P_5$ are restored up to unitary
equivalence: $P_5=0\oplus I_{2m}$,  $P_1,\dots,$ $P_4$ have the
form (\ref{eq:4tuple-2ldim}) with $\delta_1=\delta_3=0$, and
$\delta_2=\delta_4=0$ if $\lambda_{2m+1}=\alpha_5$ or
\[\delta_2=1-\delta_4=\left\{\begin{array}{l} 1, \mbox{ if }
\lambda_{2m+1}=\alpha_5+\alpha_1,\\
0, \mbox{ if } \lambda_{2m+1}=\alpha_5+\alpha_2.
\end{array}\right. \]
The parameters $p_j$, $q_j$, $r_j$, $s_j$ are calculated by
$\lambda_i$ and $\mu_i$ after the substitution
$\lambda_i=\lambda_i-\alpha_5$.

Let now $\alpha_1=\alpha_2$, then $\lambda_{2m+1}=\alpha_1+\alpha_5=\alpha_2+\alpha_5$.
Note, that
\[A_1f_{2m+1}=(\alpha_1+\alpha_5)f_{2m+1}=
(\alpha_1+\alpha_5)(P_1+P_2)f_{2m+1}.\]
 So $f'_{2m+1}=P_1f_{2m+1}$ is an eigenvector of $A_1$ too.
 Therefore we get two non-equivalent representations of $a_2$ with
 the same formulas on $P_i$ as above except the relation
\[\delta_2=1-\delta_4=\left\{\begin{array}{l} 1, \mbox{ if }
f'_{2m+1}=f_{2m+1},\\
0, \mbox{ if } f'_{2m+1}=0.
\end{array}\right. \]

The vector $f'_{2m+1}$ must be a multiple of $f_{2m+1}$ since
otherwise $f'_{2m+1}$ does not belong to the subspace (\ref{eq32})
and hence we have a new eigenvector of $A_1$ and the elements of
the sequence $\lambda_i$ are eigenvalues of $A_1$ with $i>2m+1$.
It is easy to see then that $\lambda_{2m+i}=\lambda_{2m+3-i}$ and
$\mu_{2m+i}=\mu_{2m+3-i}$. So if $\lambda_{2m_2+1}=\alpha_5$ or
$\mu_{2m_2}\in \{0,\alpha_3,\alpha_4\}$, then the relation
(\ref{eq30}) or (\ref{eq31}) does not hold. If at last
$\lambda_{2m_2+1}=0$, then the subspace

\begin{align*}\label{eq:333}
\mbox{span}(f'_{2m+1},
 D_2(\mu_{2m+2})f'_{2m+1}, D_1(\lambda_{2m+3})D_2(\mu_{2m+2})f'_{2m+1},
 \dots, \\
\notag D_2(\mu_{2m_2})D_1(\lambda_{2m_2-1})\dots
D_2(\mu_{2m+2})f'_{2m+1})
\end{align*}

 is invariant
under the act of $P_1,\dots,P_5$ and so the representation is
reducible.

({\textbf{ii}}) $\mu_{2m} \in \{0,\alpha_3,\alpha_4\}$. The
subspace
\begin{equation}\label{eq34}\mbox{span}(f_0, f_1,f_2,
\dots,  f_{2m})
\end{equation}
 is invariant
under the act of $P_1,\dots,P_5$ for $\alpha_3\ne \alpha_4$.
Operators $P_1$, $P_2,\dots,$ $P_5$ are restored up to unitary
equivalence and has the form (\ref{eq:4tuple-2lplus1dim}) with
$\delta_1=\delta_2=0$, and $\delta_3=\delta_4=0$ if $\mu_{2m}=0$
or
\[\delta_3=1-\delta_4=\left\{\begin{array}{l} 1, \mbox{ if }
\mu_{2m}=\alpha_3,\\
0, \mbox{ if } \mu_{2m}=\alpha_4,
\end{array}\right. \]
and $P_5=0\oplus I_{2m}$. Note that the parameters $p_j$, $q_j$,
$r_j$, $s_j$ are calculated here also after the substitution
$\lambda_i=\lambda_i-\alpha_5$.

As above in ({\textbf{i}}), we get two different irreducible
representations for $\mu_{2m}=\alpha_3=\alpha_4$. The formulas for
$P_i$ are the same except the relation
\[\delta_3=1-\delta_4=\left\{\begin{array}{l} 1, \mbox{ if }
f'_{2m}=f_{2m},\\
0, \mbox{ if } f'_{2m}=0,
\end{array}\right. \]
where $f'_{2m}=P_3f_{2m}$. The vector $f'_{2m}$ must be a multiple
of $f_{2m}$ since otherwise we obtain an invariant subspace with
smaller dimension or the violation in (\ref{eq30}) or
(\ref{eq31}). The proof of the fact is similar and we leave it to
the reader.

Thus we proved that with fixed coefficients $\alpha_i$ there exist
at most two non-equivalent essential representations of $a_2$. The
most simple way to construct the examples is to put
$\alpha_1=\alpha_2=\alpha_3=\alpha_4=1/2+\epsilon$,
$\alpha_5=1/(2m+5)-2\epsilon + 8\epsilon/(2m+5)$. Then
$\lambda_{2m+1}=\alpha_5$, (\ref{eq30}) and (\ref{eq31}) hold for
small irrational $\epsilon$ and hence we obtain one essential
representation of $a_2$ of dimension $2m$.  If we put
$\alpha_5=1/(4m)-2\epsilon -3\epsilon/(2m)$, then
$\lambda_{2m+1}=\alpha_1$ and we have two non-equivalent
representations of $a_2$ of dimension $2m$. It easy to see that if
$\alpha_5=1/(4m)-2\epsilon-\epsilon/(2m)$, then
$\mu_{2m}=1/2+\epsilon=\alpha_3$, that is we have two essential
representations of $a_2$ of the dimension $2m+1$ in this case. For
last case we put $\alpha_5=1/(2m)-2\epsilon$, then $\mu_{2m}=0$
and for small irrational $\epsilon$, the poset $a_2$ has the only
one up to unitary equivalence essential representation in the
dimension $2m+1$.

2) \emph{Representations of $a_4$.} Let $a_4=S_1\cup S_2$,
$S_1=\{g_1,g_2,g_5\mid g_5>g_1,g_5>g_2\}$, $S_2=\{g_3,g_4,g_6\mid
g_6>g_3,g_6>g_4 \}$ and orthoprojections $P_1,\dots, P_6$ form an
essential orthoscalar representation of $a_4$  with character $
\alpha=(\alpha_1,\alpha_2,\alpha_3,\alpha_4,\alpha_5,\alpha_6).$
The operator $A_1$ is singular so we can construct the sequences
$\lambda_i$ and $\mu_j$ of eigenvalues of $A_1$ and $A_2$ as we
did for the representations of $a_2$. Note that $A_2 $ is also
singular since otherwise $P_6$ will be the identity matrix.
Therefore there exist $m$ such that (\ref{eq30}) and (\ref{eq31})
hold for every $i,j=0,\dots,m$ and $\mu_{2m}=0$. Whence there
exists only one up to unitary equivalence representation. The
operators $P_i$ are restored by the formulas
(\ref{eq:4tuple-2lplus1dim}) where $\delta_i=0$, the parameters
$p_j$, $q_j$, $r_j$, $s_j$ are calculated by $\lambda_i$ and
$\mu_i$ after two substitutions $\lambda_i=\lambda_i-\alpha_5$ and
$\mu_i=\mu_i-\alpha_6$, $P_5=0\oplus I_{2m}$ and $P_6=I_{2m}\oplus
0$.

To find an appropriate character we set
\begin{equation}\label{eq35}
\alpha_1=\alpha_2=\alpha_3=\alpha_4=1/2+\epsilon,\
\alpha_5=\epsilon/2. \end{equation}
 Then the relation $\mu_{2m}=0$
yields $\alpha_6=1/(2m)-5\epsilon/2$. The inclusions (\ref{eq30})
and (\ref{eq31}) hold for small $\epsilon$.

3) \emph{Representations of $a_6$.} Let $a_6=S_1\cup S_2$,
$S_1=\{g_1,g_2,g_5,g_6\mid g_6>g_5>g_1, g_5>g_2\}$,
$S_2=\{g_3,g_4\}$ and orthoprojections $P_1,\dots P_6$ form an
essential orthoscalar representation of $a_6$  with character $
\alpha=(\alpha_1,\alpha_2,\alpha_3,\alpha_4,\alpha_5,\alpha_6).$
The operator $A_1$ is singular, hence we have again the sequences
$\lambda_i$ and $\mu_j$. If $\lambda_{2m+1}\ne \alpha_6$ for every
$m$, then $P_6=P_5$. Really, the operator $A_1$ is a sum of four
nonnegative operators and if $\alpha_6\notin \sigma(A_1)$, then
every nonzero number of $\sigma(A_1)$ is greater or equal to
$\alpha_5+\alpha_6$. So the ranges of $A_1$, $P_5$ and $P_6$
coincide. Whence $P_5=P_6$.

Thus $\lambda_{2m+1}= \alpha_6$ for some $m>0$ and we have only
one up to unitary equivalence essential representation of $a_6$.
The operator $P_i$, $i=1,\dots, 4$ have the form
(\ref{eq:4tuple-2ldim}) with $\delta_i=0$ and the parameters
$p_j$, $q_j$, $r_j$, $s_j$ calculated by $\lambda_i$ and $\mu_i$
after the substitution $\lambda_i=\lambda_i-\alpha_5-\alpha_6$.
The operator $P_5=0\oplus I_{2m}\oplus 0$ and $P_6=0\oplus
I_{2m+1}$.

To find the character for which the orthoscalar representation
exists we set $\alpha_i$ as in (\ref{eq35}) and
$\alpha_6=1/(2m+1)-(5m+2)\epsilon/(2m+1)$. Then $\lambda_{2m+1}=
\alpha_6$ and for small $\epsilon$ the inclusions (\ref{eq30}) and
(\ref{eq31}) hold.

4) \emph{Representations of dual to $a_i$.} All essential
representations of the posets dual to $a_2$, $a_4$ and $a_6$ can
be calculated from the described representations using duality and
the formulas $Q_g=I-P_g$.

 \vskip 0,2cm
{\it Acknowledgements.} The authors would like to thank the
professors Yu. S. Samo\v \i lenko and  S. A. Kruglyak for helpful
discussions and stimulating suggestions.


\begin{thebibliography}{99}

 \bibitem{naz-roi}
 L. A. Nazarova, A. V. Roiter.
Representations of partially ordered sets. (Russian).
\emph{Investigations on the theory of representations.
Zap. Nauchn. Sem. Leningrad. Otdel. Mat. Inst. Steklov. (LOMI)}
\textbf{28}
(1972),
pp.~5--31.


\bibitem{kleiner}
M. M. Kleiner.
Partially ordered sets of finite type. (Russian).
\emph{Investigations on the theory of representations.
Zap. Nauchn. Sem. Leningrad. Otdel. Mat. Inst. Steklov. (LOMI)}
\textbf{28}
(1972),
pp.~32--41.

\bibitem{kostya_brasil}
V. M. Bondarenko, V. Futorny, T. Klimchuk, V. V. Sergeichuk, K. Yusenko.
Systems of subspaces of a unitary space.
\emph{Linear Algebra Appl.}
\textbf{438}
(2013),
no. 5,
pp.~2561--2573.

\bibitem{krs_2002}
{S. A. Kruglyak, V. I. Rabanovich, Yu. S. Samoilenko},   
{On sums of projections}, 
\emph{  Funct. Anal. Appl.} 
\textbf{36} 
{(2002)},   
{no.~3},    
pp.~{182--195 }. 

\bibitem{kru_roi}
S. A. Kruglyak, A. V. Roiter.
Locally scalar representations of graphs in the category of Hilbert spaces. (Russian)
\emph{Funktsional. Anal. i Prilozhen.}
\textbf{39}
(2005),
no. 2,
pp.~13--30;
translation in
\emph{Funct. Anal. Appl.} \textbf{39} (2005), no. 2, pp.~91--105

\bibitem{alb_ost_sam}
{S. Albeverio, V. Ostrovskyi, Yu. Samoilenko},   
{On functions on graphs and representations of a certain class of $*$-algebras}, 
\textit{J. Algebra.} 
\textbf{308},
{(2007)},   
{567--582}. 

\bibitem{kostya_uni}
T. Weist, K. Yusenko.
Unitarizable Representations of Quivers.
\emph{Algebr. Represent. Theory}
\textbf{16}
(2013),
no.~5,
pp.~1349--1383.

\bibitem{kostya_uni2}
Yu. Samoilenko, K. Yusenko.
Kleiner's theorem for unitary representations of posets.
\emph{Linear Algebra Appl.}
\textbf{437}
 (2012),
no.~2,
pp.~581--588.


\bibitem{ost-sam99}
{V. Ostrovskyi, Yu. Samoilenko},   
\emph{Introduction to the Theory of
Representations of Finitely Presented $*$-Algebras. I. Representations by Bounded Operators}, 
{Rev. Math. Math. Phys.} 
\textbf{11} 
{Harwood Academic Publishers},    
{Amsterdam},     
{1999}.     

\bibitem{kyrych}
A. A. Kirichenko.
On linear combinations of orthoprojectors,
\emph{Uch. Zap. Tavrich. Univ.}
\textbf{54} (2002),
no. 2,
pp.~31--39.

\bibitem{yusenko-four}
K. A. Yusenko.
On quadruples of projectors connected by a linear relation. (Ukrainian)
\emph{Ukrain. Mat. Zh.}
\textbf{58}
(2006),
 no. 9,
pp.~1289--1295;
translation in
\emph{Ukrainian Math. J.} \textbf{58} (2006), no. 9, pp.~1462--1470.


\bibitem{ost-sam06}
{V. L. Ostrovskyi and Yu. S. Samoilenko}, 
{ On spectral theorems for families of linearly connected
selfadjoint operators with prescribed spectra associated with
extended Dynkin graphs} (Ukrainian). 
\emph { Ukrain  Mat. Zh. } 
\textbf{58} 
 (2006), 
  {no.~11,} 
 pp.~{1556--1570}. 

\bibitem{ost05}
V. Ostrovskyi.
On $∗$-representations of a certain class of algebras related to a graph.
\emph{Methods Funct. Anal. Topology}
\textbf{11}
(2005),
no. 3,
pp.~250--256.

\bibitem{otr}
V.V.Otrashevskaya. On representations of one-parameter partially
ordered sets, in "Matrix problems" Inst. Mat. Akad. Nauk Ukrain.
SSR, Kiev (1977), pp.~144-149.

\bibitem{kru_naz_roi}
S. A. Kruglyak,  L. A. Nazarova, A. V. Roiter.
Orthoscalar representations of quivers in the category of Hilbert spaces.
\emph{J. Math. Sci.}
 \textbf{145}
(2007),
no. 1,
pp.~4793--4804.

\end{thebibliography}
\end{document}